\documentclass[12pt]{amsart}
\usepackage{amsfonts, amscd,color,amsmath,amssymb,graphics,graphicx}
 \DeclareMathOperator{\diam}{diam}
\DeclareMathOperator{\mesh}{mesh}
 
\DeclareMathOperator{\inte}{int} \DeclareMathOperator{\bd}{bd}

\DeclareMathOperator{\st}{St}
\DeclareMathOperator{\barycentrum}{b}

\newtheorem{theorem}{Theorem}[section]
\newtheorem{lemma}[theorem]{Lemma}
\newtheorem{corollary}[theorem]{Corollary}
\newtheorem{proposition}[theorem]{Proposition}

\newtheorem{fact}[theorem]{Fact}

\theoremstyle{definition}
\newtheorem{definition}[theorem]{Definition}
\newtheorem{example}[theorem]{Example}

\theoremstyle{remark}
\newtheorem{remark}[theorem]{Remark}
\newtheorem{question}[theorem]{Question}

\begin{document}
\title[Chain recurrent sets]{Chain recurrent sets of generic mappings on compact spaces}

\author[P. Krupski]{Pawe\l\ Krupski}
\email{Pawel.Krupski@math.uni.wroc.pl}

\author[K. Omiljanowski]{Krzysztof Omiljanowski}
\email{Krzysztof.Omiljanowski@math.uni.wroc.pl}

\author[K. Ungeheuer]{Konrad Ungeheuer}
\email{Konrad.Ungeheuer@math.uni.wroc.pl}
\address{Mathematical Institute, University of Wroc\l aw, pl.
Grunwaldzki 2/4, 50--384 Wroc\l aw, Poland}
\date{\today}
\subjclass[2010]{Primary 37B45; Secondary 54F15}
\keywords{ANR-space, chain recurrent point, continuum, generic map, $LC^{n}$-space, $n$-dimensional map, local-periodic-point property, periodic point, polyhedron, retraction}

\begin{abstract} Let 0-$CR$ denote the class of all metric compacta $X$ such that the set of maps $f:X\to X$ with 0-dimensional sets $CR(f)$ of chain recurrent points is a dense $G_\delta$-subset of the mapping space $C(X, X)$ (with the uniform convergence). We prove, among others,   that  countable products of polyhedra or locally connected curves belong to  0-$CR$. Compacta that admit, for each $\epsilon>0$, an $\epsilon$-retraction onto a subspace from 0-$CR$  belong to 0-$CR$ themselves.  Perfect ANR-compacta or $n$-dimensional $LC^{n-1}$-compacta have perfect  $CR(f)$ for a generic self-map $f$. In the cases of polyhedra, compact Hilbert cube manifolds, local dendrites and their finite products, a generic  $f$ has $CR(f)$ being a Cantor set and the set of periodic points of $f$ of arbitrarily large periods is dense in $CR(f)$. The results extend some known facts about $CR(f)$ of  generic self-maps $f$ on PL-manifolds.
\end{abstract}

\maketitle

\section{Introduction}
  For compact metric spaces  $X, Y$, we denote by $C(X,X)$  the space of all continuous maps with the topology of uniform convergence. We say that a map $f\in C(X,X)$ with a property $\mathcal P$ is \emph{generic} if the set of all maps in $C(X,X)$ with property $\mathcal P$ is residual (i.~e., contains a dense  $G_\delta$-subset) in $C(X, X)$.

 In this paper we investigate chain recurrent points of generic maps on some compacta. Recall that, given a map $f:(X,d)\to (X,d)$ and $\epsilon>0$,  a finite set $\{x=x_0, x_1,\dots, x_n=y\}\subset X$ is an \emph{$\epsilon$-chain} for $f$ from $x$ to $y$ if $d(f(x_{i-1}),x_i)<\epsilon$ for each $i=1,\dots, n$;  a point $x$ is called a \emph{chain recurrent point} of $f$ if, for each $\epsilon>0$, there exists an $\epsilon$-chain for $f$ from $x$ to $x$ (sometimes, $x$ is said to have a periodic  $\epsilon$-pseudo-orbit for each $\epsilon$). The set of chain recurrent points of $f$ is denoted by $CR(f)$. Clearly, it contains the set $Per(f)$ of  periodic points of $f$.

The notion of a chain recurrent point plays an important role in dynamical systems. Some general properties of it can be found, e.g., in books~\cite{A}, \cite{AH}, \cite{BC}.  There is an extensive literature devoted to generic properties of  maps involving chain recurrency and relative notions, where they were studied mainly for diffeo- or homeomorphisms on smooth or PL-manifolds  (see, e.g., \cite{PPSS}, \cite{Pil}, \cite{Hur1}, \cite{Hur2}, \cite{Mazur}, \cite{AHK}).

Denote by  0-$CR$ the family of all compacta $X$ such that the set $CR(f)$ is 0-dimensional for a generic map $f\in C(X,X)$. It is known that all finite graphs and PL-manifolds are in class 0-$CR$ (see~\cite{Y} and~\cite{AHK}, resp.). Generalizing these results, we prove in Section~\ref{poly} that all finite polyhedra belong to class 0-$CR$. In Section~\ref{LC}, we show that all locally connected, 1-dimensional continua are elements of 0-$CR$, as well. Moreover, using a small-retractions-technique, we derive similar results for all Menger compacta, compact Hilbert cube manifolds and some important non-locally connected continua. The technique also gives us immediately that all countable (finite or infinite) products  as well as  cones and suspensions over these compacta are members of 0-$CR$.

A stronger property of a generic map $f$, of having $CR(f)$ homeomorphic to the Cantor set $\mathcal C$, is formulated in~\cite{AHK} for PL-manifolds (the proof in~\cite{AHK}, however, is only barely sketched at the very end of the memoir).
In Section~\ref{per} we provide a short  proof that  $CR(f)$ has no isolated points for a generic map $f:X\to X$ if $X$ is a perfect ANR or $X$ is a perfect $n$-dimensional $LC^{n-1}$-compactum; moreover, if such $X$ has the  local periodic point property, then $Per(f)$ has no isolated points and is dense in  $CR(f)$; moreover, for each integer $l\ge 2$, periodic points of periods $\ge l$ form a dense subset of  $CR(f)$.  In particular, if $X$ is a finite polyhedron, a compact Hilbert cube manifold, a local dendrite or a finite product of these spaces, then $CR(f)$ is a Cantor set containing $Per(f)$ as a dense subset for a generic $f$.

Sometimes we can also claim that a generic self-map is  zero-dimen\-sio\-nal. This is true in the cases of ANR-compacta with coinciding finite dimensions $\dim = ped$, where $ped$ is the piecewise embedding dimension in the sense of~\cite{KM}, Menger manifolds and countable products of locally connected curves, among others.

Finally, we observe that a measure-theoretical result of~\cite{Y} that  $CR(f)$ is of measure zero for a generic map $f\in C(X,X)$, if $X$ is an $n$-dimensional $LC^{n-1}$ perfect compactum equipped with a Borel, finite, non-atomic measure, extends on all mentioned above compacta.

\section{Some general properties}
The following characterization of  chain recurrent points is very useful.
 \begin{proposition}\cite{BF}\label{proBF} For any compact space $X$ and $f\in C(X,X)$,
 $x\notin CR(f)$ if and only if there exists an open set $U\subset X$ such that $x\notin U$, $f(x)\in U$ and $f(\overline{U})\subset U$.
 \end{proposition}

 It follows easily from this proposition that (for each compact space $X$ and $f\in C(X,X)$),
 the set $CR(f)$ is  nonempty, closed and invariant under $f$~\cite{BC}.

We will use the following fact.

\begin{fact}\cite{BF}\label{f1}
The set-valued function
$$CR: C(X,X)\to 2^X, \quad f\mapsto CR(f)$$
defined on a compact metric space $X$, is upper semi-continuous.
\end{fact}
Here, $2^X$ is the hyperspace of all nonempty closed subsets of $X$ with the Hausdorff metric.

\

Recall the notion of the Urysohn coefficient $d_n(X)$ of a metric space $X$, used in dimension theory~\cite[Ch. IV, \S 45, IV]{Kur}:

$d_n(X)$ is the infimum of all $\epsilon>0$ such that there exists an open cover $\{G_0,\dots, G_m\}$ of $X$ such that $\diam G_i<\epsilon$ for each $i$ and
$$G_{i_0}\cap\dots\cap G_{i_n}=\emptyset \quad\text{for any}\quad i_0<\dots < i_n\le m. $$

In particular, if $X$ is compact, then $d_1(X)<\epsilon$ if and only if each component of $X$ has diameter $< \epsilon$.

\begin{lemma}\label{le1}
For any compact  space $X$, the set
$$\{f\in C(X,X): \dim (CR(f))=0\}$$
 is a $G_\delta$-subset of $C(X, X)$.
\end{lemma}

\begin{proof}

It suffices to prove that the set
$$A_n=\{f\in C(X,X): d_1(CR(f))<\frac1n\}$$ is open in $C(X,X)$, since
$$\{f\in C(X,X): \dim  (CR(f))=0\}=\bigcap_n A_n.$$

So, let $\{f_i\}\subset  C(X, X)\setminus A_n$ and $\lim_i f_i = f$. There is a component $C_i$ of $CR(f_i)$ such that $\diam C_i\ge \frac1n$ for each $i$. Without loss of generality, assume the sequence $\{C_i\}$ converges (in the hyperspace $C(X)$ of subcontinua of $X$) to a continuum $C$.
Then $\diam C\ge\frac1n$ and, by Fact~\ref{f1},  $C\subset CR(f)$. Hence $f\in  C(X, X)\setminus A_n$.

\end{proof}

\begin{lemma}\label{le2}
If $X$ is a compact space, $r:X\to G$ is a retraction and $f\in C(G,G)$, then $CR(f)=CR(fr)$.
\end{lemma}

\begin{proof}
The inclusion $CR(f) \subset CR(fr)$ is obvious. So, assume  $x \notin CR(f)$. If $x \in X \setminus G$, then applying Proposition~\ref{proBF}   for $fr$ and $U=X\setminus \{x\}$, we get $x \notin CR(fr)$. If $x\in G$, then, by the same proposition, there is an open subset $U$ of $G$ such that $x \notin U$, $f(x) \in U$, $f(\overline{U}) \subset U$. Put $V= r^{-1}(U)$. Then
\begin{multline*}
x \notin V, \quad fr(x) = f(x) \in U \subset V \quad\text{and}\\
 fr( \overline{V}) = f(r( \overline{V})) \subset f( \overline{r(V)})=
f( \overline{U}) \subset U \subset V.
\end{multline*}
Again, by Proposition~\ref{proBF}, we obtain $x \notin CR(fr)$.
\end{proof}

We are now ready to formulate an approximation theorem.

\begin{theorem}\label{t:approx}
 Let $X$ be a compact space such that for each  $\epsilon >0$ there is a retraction $r_\epsilon :X \rightarrow Y_\epsilon$,  $\hat{d}(r_\epsilon, id) < \epsilon$ and  $Y_\epsilon=r_\epsilon(X)\in$  0-$CR$. Then  $X\in$ 0-$CR$.
\end{theorem}

\begin{proof}
In view of Lemma~\ref{le1}, it remains to show that the set
$$\{ f \in C(X,X): \dim(CR(f))=0  \}$$
is dense in $C(X,X)$.

Fix  $\epsilon > 0$, $f \in C(X,X)$ and  $0< \gamma <\epsilon /4$ such that  $d(x,y)<\gamma$ implies $d(f(x), f(y))< \epsilon /4$. Let $r_\gamma$ be a retraction as in the hypothesis. Then
$$\hat{d}(r_\gamma fr_\gamma , \ f)< \epsilon /2,$$
and
$$\bigl((r_\gamma f)|{Y_\gamma}\bigr) r_\gamma = r_\gamma f r_\gamma.$$
Since the map $(r_\gamma f)|{Y_\gamma} : Y_\gamma \rightarrow Y_\gamma$ belongs to  $C(Y_\gamma,Y_\gamma)$, there exists $g \in C(Y_\gamma,Y_\gamma)$ such that
 $$\hat{d}\bigl((r_\gamma f)|{Y_\gamma}, g\bigr)< \epsilon /4,\quad\text{and}\quad \dim(CR(g))=0.$$

 We have
 $\hat{d}\Bigl(\bigl((r_\gamma f)|{Y_\gamma}\bigr) r_\gamma, gr_\gamma \Bigr)< \epsilon /4,$
  therefore
  $$\hat{d}(f, gr_\gamma) \leq \hat{d}(f, r_\gamma fr_\gamma) + \hat{d}(r_\gamma fr_\gamma, gr_\gamma) < \epsilon /2 + \epsilon /4 < \epsilon.$$
It also follows from Lemma~\ref{le2} that   $\dim CR(gr_\gamma)=0$, so the density is proved.
\end{proof}

\begin{example}
Here is a quick example of an application of Theorem~\ref{t:approx}: the standard $\sin\frac1x$-curve belongs to 0-$CR$.
\end{example}

\

The following simple technical lemma will be helpful in estimating a size of the chain recurrent set $CR(f)$.

\begin{lemma}\label{le3}
Let $X$ be a compact space and $f \in C(X,X)$ be such that there exist families $\{U_i \}_{i=0}^{n} $ of open subsets of $X$ and $\{K_i \}_{i=0}^{n} $ of closed subsets of $X$ satisfying
$$\begin{array}{ll}
                          & U_0 \,  \subset \, U_1  \subset  \dots  \subset \, U_{n-1}  \subset  U_n,\\
			  & \, \cup  \ \ \ \ \  \cup \  \ \ \ \ \ \ \    \ \ \ \ \  \cup \ \ \ \ \ \ \, \cup \\
			  & K_0   \subset  K_1  \subset  \dots  \subset K_{n-1} \subset K_n,\\
         \end{array}$$
and $f(U_i) \subset K_i$ for  $i=0, \dots, n$. Then
 $$CR(f) \subset \bigl((X \setminus  U_n) \cup (K_n \setminus U_{n-1}) \cup (K_{n-1} \setminus U_{n-2}) \cup \dots \cup K_0 \bigr).$$
\end{lemma}

\begin{proof}
For $x \in U_i \setminus K_i$ and $V= U_i \setminus \{ x \}$,
we have
$$x \notin V, \quad f(x) \in V, \quad f(\overline{V}) = f(\overline{U_i \setminus \{ x \} }) \subset \overline{f(U_i \setminus \{ x \} ) } \subset \overline{K_i} = K_i \subset V,$$
so, by Proposition~\ref{proBF}, $CR(f) \cap U_i \subset K_i$. Hence
$$ CR(f) \cap (U_i \setminus U_{i-1}) =  (CR(f) \cap U_i) \setminus U_{i-1} \subset K_i \setminus U_{i-1}.$$
Since the family $\{U_i\}_{i=0}^{n}$ is increasing, we can represent $X$ as the union
 $$X=(X \setminus  U_n) \cup (U_n \setminus U_{n-1}) \cup (U_{n-1} \setminus U_{n-2} ) \cup \dots \cup U_0$$
 and the conclusion follows.
 \end{proof}

\

\section{Polyhedra}\label{poly}
In this section we will prove that class 0-$CR$ contains all polyhedra.
A simplex $S$  with vertices  $a_0, a_1, \dots , a_n$ in a Euclidean space will be denoted by  $S=a_0 a_1 \dots  a_n$.
If  $A$ is a face of a simplex $S$, then we write  $A \leq S$
( $A < S$, if $A\neq S$).
The barycenter of $S$ is denoted by  $\barycentrum(S)$.

We will consider only finite simplicial complexes.
If $\mathcal{W}$ is a simplicial complex, then    ${\mathcal{W}^{[i]}}$  denotes its $i$-skeleton, $\mathcal{W}',\mathcal{W}'',\dots, \mathcal{W}^{(m)}$ are subsequent barycentric subdivisions of $\mathcal{W}$,  $\mesh(\mathcal{W}):= \max \{\diam (S): S \in \mathcal{W} \}$.
$|\mathcal{W}|:= \bigcup \mathcal{W}$ is a polyhedron and each simplicial complex  $\mathcal{K}$ such that $|\mathcal{W}|=|\mathcal{K}|$ is called a triangulation of  $|\mathcal{W}|$.

Given a subcomplex
$\mathcal{P}$ of $\mathcal{W}$, the star of $\mathcal{P}$ in  $\mathcal{W}$ is the set
$$\st_\mathcal{K} \mathcal{P} = |\mathcal{K}| \setminus \bigcup  \{ S \in \mathcal{K} : S \cap |\mathcal{P}| = \emptyset   \} .$$

Any simplicial vertex map $f^{[0]}:  {\mathcal{K}^{[0]}} \rightarrow {\mathcal{W}^{[0]}}$ between vertices of complexes $\mathcal{K}$ and  $\mathcal{W}$ determines   simplicial maps $f :\mathcal{K} \rightarrow \mathcal{W}$ between complexes and $|f|:|\mathcal{K}|\rightarrow |\mathcal{W}|$ between corresponding polyhedra.

\

The following lemma is crucial for this section.

\begin{lemma}\label{le4}
If $\mathcal{K}$ is a simplicial complex, then  there is a simplicial vertex map  $g^{[0]}: {(\mathcal{K}'')^{[0]}} \rightarrow {(\mathcal{K}')^{[0]}}$ such that
$$\hat{d}(|g|, id) \leq \mesh(\mathcal{K}) \quad\text{and}\quad d_1(CR(|g|)) \leq \mesh(\mathcal{K}).$$
\end{lemma}

\begin{proof}
Suppose $\dim |\mathcal{K}|=n$. Represent the first barycentric subdivision of $\mathcal{K}$ in the canonical way:
\begin{multline*}
\mathcal{K}'=\\
\{\barycentrum(S_0)\barycentrum(S_1) \dots \barycentrum(S_k): S_0 > S_1 > \dots > S_k,\,  S_0,\dots,S_k \in \mathcal{K},\, k \leq n \}.
\end{multline*}
Observe that
\begin{multline*}
{(\mathcal{K}'')^{[0]}}=\\
\{\barycentrum\bigl(\barycentrum(S_0)\barycentrum(S_1) \dots \barycentrum(S_k)\bigr):S_0 > S_1 > \dots > S_k,\, S_0, \dots , S_k \in \mathcal{K},\, k \leq n  \}.
\end{multline*}
Define  $g^{[0]}: {(\mathcal{K}'')^{[0]}} \rightarrow {(\mathcal{K}')^{[0]}}$  by
$$g^{[0]}\bigl(\barycentrum(\barycentrum(S_0)\barycentrum(S_1) \dots \barycentrum(S_k))\bigr)=\barycentrum(S_k),$$
where $S_0 > S_1 > \dots > S_k$, $S_0,  \dots , S_k \in \mathcal{K}$.

Since each vertex $x \in (\mathcal{K}'')^{[0]}$ is uniquely represented in the form
$$x=\barycentrum\bigl(\barycentrum(S_0)\barycentrum(S_1) \dots \barycentrum(S_k)\bigr),$$
the map $g^{[0]}$ is well defined  (Figure~\ref{Fig1}).

\begin{figure}[h]
\includegraphics[width=102mm]{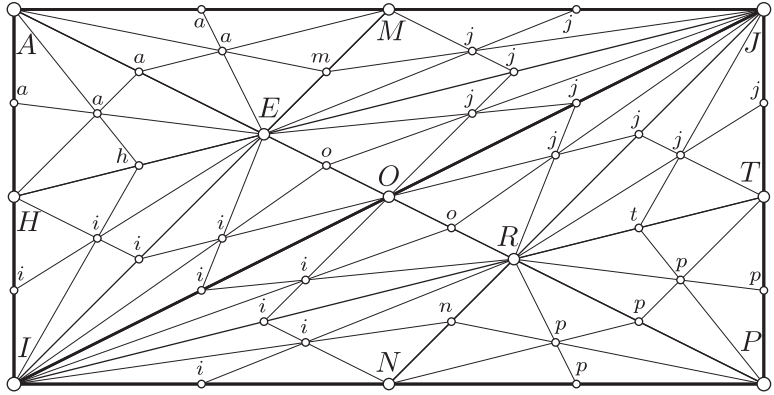}
\caption{The action of $g$ on $\mathcal{K}''$ for a rectangle (simplices of $\mathcal{K}$ are drawn with thick lines). Vertices in $(\mathcal{K}'')^{[0]}$  are labeled by capital and small letters $X, x$ so that   $g^{[0]}(x) = X$ and  $g^{[0]}(X)= X$.}\label{Fig1}
\end{figure}

Moreover, $g^{[0]}$ is a simplicial vertex map. Indeed, if $a_0 a_1 \dots a_t \in \mathcal{K}''$, where
$$a_i=\barycentrum\bigl(\barycentrum(S_0^i)\barycentrum(S_1^i) \dots \barycentrum(S_{k(i)}^i)\bigr), \quad  i=0,\dots ,t,$$
then there is a simplex  $S \in \mathcal{K}'$ such that
$$\barycentrum(S_0^i)\barycentrum(S_1^i) \dots \barycentrum(S_{k(i)}^i) \leq  S,\quad i=0, \dots ,t,$$
in particular $\barycentrum(S_{k(i)}^i)$ is a vertex of $S$. Hence, the set
$$\{ \barycentrum(S_{k(0)}^0), \barycentrum(S_{k(1)}^1), \dots  ,\barycentrum(S_{k(t)}^t)\}=g^{[0]}\bigl((a_0 a_1 \dots a_t)^{[0]}\bigr)$$ spans a face of $S$, so it is a simplex in $\mathcal{K}'$.

Notice that $|g|(\st_{\mathcal{K}''}\barycentrum(S_0))\subset |S_0|$ for each $S_0 \in \mathcal{K}$. The closure of $\bigcup \{\st_{\mathcal{K}''}\barycentrum(S): S \leq S_0\}$ contains  $|S_0|$ in its interior for each $S_0 \in \mathcal{K}$ (Figure~\ref{Fig2}).

\begin{figure}[h]
\centering
\includegraphics[width=107mm]{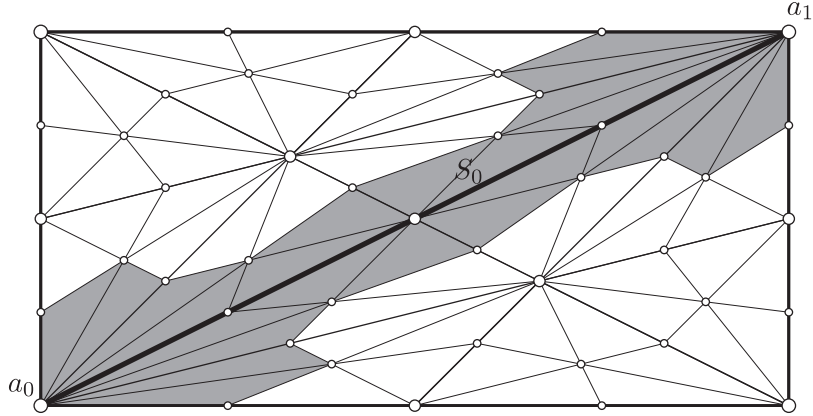}
\caption{The shadowed area shows the closure of $\bigcup \{\st_{\mathcal{K}''}b(S):S \leq S_0\}$ which is mapped by $|g|$ onto $S_0=a_0 a_1$.}\label{Fig2}
\end{figure}

Let
$$ W_i = \bigcup_{S_0 \in \mathcal{K}^{[i]}} \bigcup \{ \st_{\mathcal{K}''}\barycentrum(S):S \leq S_0 \}.$$
 The set $U_i = \inte \overline{W_i}$ is open in $|\mathcal{K}|$, contains $|\mathcal{K}^{[i]}|$ and $|g|(U_i) = |\mathcal{K}^{[i]}|$ for $i=0, \dots, n$ ($|\mathcal{K}^{[n]}|=|\mathcal{K}|=U_n$).

The map $|g|$, families  $\{U_i\}_{i=0}^{n-1}$ and  $\{|\mathcal{K}^{[i]}|\}_{i=0}^{n-1}$ satisfy hypotheses of Lemma~\ref{le3}. Hence,
$$CR(|g|) \subset (|\mathcal{K}|\setminus U_{n-1}) \cup (|\mathcal{K}^{[n-1]}|\setminus U_{n-2}) \cup (|\mathcal{K}^{[n-2]}|\setminus U_{n-3}) \cup \dots \cup |\mathcal{K}^{[0]}|.$$

Since $|\mathcal{K}^{[i-1]}|$ separates $|\mathcal{K}^{[i]}|$, the set $U_{i-1}$ separates $U_{i}$, hence
$$d_1(|\mathcal{K}^{[i]}| \setminus U_{i-1}) \leq
d_1(|\mathcal{K}^{[i]}| \setminus |\mathcal{K}^{[i-1]}|) \leq
 \mesh (\mathcal{K}), \quad i=0, \dots , n.$$
Thus $d_1(CR(|g|)) \leq \mesh (\mathcal{K})$.
Clearly, we have  $\hat{d} (|g|, id) \leq \mesh (\mathcal{K})$.

\end{proof}

\begin{lemma}\label{le5}
Let $\mathcal{L}$ be a simplicial complex. Then for each simplicial vertex map $f^{[0]}:(\mathcal{L}^{(k+1)})^{[0]} \rightarrow \mathcal{L}^{[0]}$ there is a simplicial vertex map  $h^{[0]} :(\mathcal{L}^{(k+2)})^{[0]} \rightarrow \mathcal{L}^{[0]}$ such that
$$d_1(CR(|h|)) \leq \mesh(\mathcal{L})\quad\text{and}\quad \hat{d}(|f|, |h|) \leq \mesh(\mathcal{L}).$$
\end{lemma}

\begin{proof}
Suppose $\dim |\mathcal{L}|=n$. Let $g^{[0]}:{(\mathcal{L}^{(k+2)})^{[0]}} \rightarrow {(\mathcal{L}^{(k+1)})^{[0]}}$ be a map as in Lemma~\ref{le4} and $h^{[0]}:=f^{[0]}g^{[0]}$. We have
$$|h|(|(\mathcal{L}^{(k)})^{[i]}|) \subset |\mathcal{L}^{[i]}| \subset |(\mathcal{L}^{(k)})^{[i]}| \quad\text{for}\quad i=0, \dots , n.$$
Put
$$W_i= \bigcup_{S_0 \in (\mathcal{L}^{(k)})^{[i]}} \bigcup \{ \st_{\mathcal{L}^{(k+2)}}\barycentrum(S):S \leq S_0 \}.$$
The map $|h|$ and families  $\{|(\mathcal{L}^{(k)})^{[i]}|\}^{n-1}_{i=0}$ and $\{U_i\}^{n-1}_{i=0}$, where $U_i = \inte \overline{W_i}$ satisfy hypotheses of Lemma~\ref{le3}, since
$$|h|(U_i)=|f|(|g|(U_i))=|f|(|(\mathcal{L}^{(k)})^{[i]}|) \subset |\mathcal{L}^{[i]}| \subset |(\mathcal{L}^{(k)})^{[i]}|.$$
Now, we proceed similarly to the proof of Lemma~\ref{le4}. Namely, we have
$$CR(|h|) \subset (|(\mathcal{L}^{(k)})|\setminus U_{n-1}) \cup (|(\mathcal{L}^{(k)})^{[n-1]}|\setminus U_{n-2}) \cup \dots \cup |(\mathcal{L}^{(k)})^{[0]}|$$
and $|(\mathcal{L}^{(k)})^{[i-1]}|$ separates $|(\mathcal{L}^{(k)})^{[i]}|$. This yields  inequalities
$$d_1(|(\mathcal{L}^{(k)})^{[i]}| \setminus U_{i-1}) \leq  d_1(|(\mathcal{L}^{(k)})^{[i]}| \setminus |(\mathcal{L}^{(k)})^{[i-1]}|)  \leq \mesh (\mathcal{L}^{(k)}),$$
for  $i=0, \dots , n$. So, if $C$ is a component of  $CR(|h|)$, $S \in \mathcal{L}^{(k)}$ and $C \cap S \neq \emptyset$, then $C \cap A = \emptyset$  for $A < S$ . Hence
 $$d_1(CR(|h|)) \leq \mesh (\mathcal{L}^{(k)}) \leq \mesh(\mathcal{L}).$$
Clearly,  $\hat{d} (|h|, |f|) \leq \mesh (\mathcal{L})$.
 \end{proof}

\begin{theorem}\label{t:poly}
Polyhedra belong to class 0-$CR$.
\end{theorem}

\begin{proof}
Let $X$ be a polyhedron. In view of Lemma~\ref{le1} it suffices to show that the set
$$A_m = \left\{ f \in C(X,X): d_1 (CR(f)) < 1/m \right\}$$ is dense in $C(X,X)$ for each $m$.
Fix arbitrary $F \in C(X,X)$ and $0< \epsilon < 1/m$.
Choose a triangulation $\mathcal{L}$ of $X$  with $\mesh(\mathcal{L}) < \epsilon$ and a simplicial approximation $f^{[0]}:(\mathcal{L}^{(k+1)})^{[0]}\rightarrow \mathcal{L}^{[0]}$ of $F$ such that $\hat{d}(|f|, F) \leq \mesh(\mathcal{L})$.
By Lemma~\ref{le5}, there exists $h^{[0]} :(\mathcal{L}^{(k+2)})^{[0]} \rightarrow \mathcal{L}^{[0]}$ such that $\hat{d} (|h|,|f|)< \mesh(\mathcal{L})$ and  $d_1(CR(|h|))< \mesh(\mathcal{L})$. Hence, the map  $H:= |h|$ satisfies $\hat{d} (F, H) \leq 2 \mesh(\mathcal{L}) < 2 \epsilon$ and $H \in A_m$.
\end{proof}

\

\section{Locally connected curves}\label{LC}
By a curve we mean a one-dimensional continuum. In this section we present two different proofs of the following theorem.

\begin{theorem}\label{t:LCcurves}
Locally connected curves are in  class 0-$CR$.
\end{theorem}

The first proof (inspired by an argument presented in~\cite{Y}) is based on the  well known fact recalled below.
\begin{lemma}[{\cite[p. 80]{Bor}}] \label{le0}
If a compact space $X$ is an ANR-space (or is $n$-dimensional $LC^{n-1}$-space for some  $n \in \mathbb{N}$), then for each $\epsilon >0$ there is $\delta>0$ such that for each continuous map  $\varphi :A\rightarrow X$ from a closed subset $A$ of a compact space  $Z$ (with $\dim (Z \setminus A) \leq n$) satisfying $\diam (\varphi (A))<\delta$, there is a continuous extension  $\psi:Z \rightarrow X$ of  $\varphi$ such that $\diam(\psi (Z)) < \epsilon$.
\end{lemma}
\begin{proof}[Proof 1 of Theorem~\ref{t:LCcurves}]
 Suppose $X$ is a locally connected curve,  $f \in C(X,X)$ and $\epsilon >0$. Since $X\in LC^0$, we can apply Lemma~\ref{le0} to choose $\delta >0$  for $\epsilon / 4$.
Choose also $0<\xi<\delta /2 <  \epsilon /4$  so that if $d(x,y)<\xi$, then  $d(f(x), f(y))<\delta / 2$.

We are going to construct a map $g \in C(X,X)$ such that $\hat{d}(f, g) < \epsilon$ and each component of  $CR(g)$ has diameter less than $\epsilon$. Then, by Lemma~\ref{le1}, the proof will be complete.

Since $X$ is compact, one dimensional, there exists a finite open cover
 $\mathcal{V}$ of $\mesh < \xi$ such that  $\dim (\bd V) = 0$ for $V\in \mathcal{V}$.
The subspace  $B = \bigcup_{V \in \mathcal{V}} \bd(V)$ is zero-dimensional compact, so there is a finite cover $\mathcal{U} =\{U_1, \dots , U_p  \}$ of $B$ such that $\mathcal{U}$ refines  $\mathcal{V}$, the sets $U_i$  are pairwise disjoint open subsets of $X$ and  $\bd (U_i) \cap B = \emptyset$ for  $i = 1, \dots , p$.
 For each $U_i$ choose a pair $(z_i, W_i)$ such that $W_i$ is an open subset of $X$ and
  $$z_i \in W_i \cap B = U_i \cap B \subset W_i \subset \overline{W_i} \subset U_i.$$
The set $\{ z_i\}_{i=1}^p$ is $\xi$-dense in  $X$. So, for any  $i \in \{1, \dots , p\}$, we can choose an index $m(i) \in \{1, \dots , p\}$ such that $d(f(z_i), z_{m(i)})< \xi$.
Let
$$g(x) = \left\{ \begin{array}{ll}
         f(x), & \mbox{gdy $x \in X\setminus ({\bigcup \mathcal{U}})$};\\
			z_{m(i)}, & \mbox{gdy $x \in \overline{W_i}$}, \quad  i=1, \dots ,p. \end{array} \right. $$
Since $\diam (f(U_i))< \delta / 2$ for $ i=1, \dots ,p$, we have
$$\diam(g(\bd(U_i)\cup \overline{W_i}))< \delta /2 + \delta /2< \delta.$$
By Lemma~\ref{le0}, $g$   extends over each  $U_i \setminus \overline{W_i}$ so that
$$\hat{d}(f, g) < \delta + \xi + \epsilon /4 < \epsilon.$$
It follows from Proposition~\ref{proBF} that
$$CR(g) \cap (W_i\setminus \{ z_i \}) =  \emptyset \quad\text{for}\quad i=1, \dots ,p.$$
Thus no component of $CR(g)$ has diameter greater than $\epsilon$, since otherwise it would intersect $\bd(V)$ for some $V \in \mathcal{V}$, so  it would intersect some $W_i\setminus \{ z_i \}$.
\end{proof}

\

A second  proof is based on the following interesting theorem which first appeared in~\cite{Ma}. For the reader's convenience, we  provide here its short, independent proof. The theorem strengthens an old fact due to S. Mazurkiewicz~\cite{Maz} that any locally connected curve $X$ admits arbitrarily small mappings onto graphs contained in $X$. By a \emph{graph} we mean a space homeomorphic to a connected 1-dimensional polyhedron.

\begin{theorem}\label{t:curve retr}
If $X$ is a locally connected curve, then for each $\epsilon >0$ there is a graph $G_\epsilon \subset X$ and a retraction $r_\epsilon : X \rightarrow G_\epsilon$ such that $\hat{d}(r_\epsilon ,id ) < \epsilon$.
 \end{theorem}

 \begin{proof}
We  use a version of the Bing-Moise's Brick Partition Theorem  for locally connected curves as formulated in~\cite[Theorem 25]{MOT}:
\begin{quote}
\emph{Each locally connected curve has a decreasing sequence of order-two brick partitions whose elements have zero-dimensional boundaries.}
\end{quote}
This means that for each $\epsilon >0$ there is a closed cover $\{B_1, B_2,  \dots ,B_n\}$ of $X$ of order two (i.~e., at most two elements of the cover intersect) such that
$$\diam B_i < \epsilon,\quad \dim (\bd B_i ) =0,\quad \inte B_i \cap \inte B_j = \emptyset \quad\text{if}\quad i \neq j$$  and $\inte B_i $ is connected and uniformly locally connected for each  $i$ (in particular, $\inte B_i$ is arcwise connected).

Points of  $\bd B_i $ are arcwise accessible  from $\inte{B_i}$ (see~\cite[Theorem 7, p. 266]{Kur}).
 Call two distinct indexes   $i, j \in \{1,2, \dots , n   \}$ dependent if $B_i \cap B_j \neq \emptyset$. For each $i\le n$ and any pair  $(i, j)$  of dependent indexes choose
  \begin{itemize}
  \item a point $a_i\in \inte B_i$,
  \item a point $b(i,j)=b(j,i)\in B_i\cap B_j$,
    \item an arc  $\alpha(i,j)\subset \inte B_i\cup \{b(i,j)\}$ from   $a_i$ to   $b(i,j)$.
  \end{itemize}
 Then, for each $i \in \{1, 2, \dots, n\}$, choose a tree
   $$T_i \subset \bigcup \{ \alpha(i, j): \text{$i, j$ are dependent}\}$$
   with ends $b(i,j)$ for all  $j$'s dependent on   $i$ (Figure~\ref{Fig3}).

\begin{figure}[h]
\centering
\scalebox{1}[.6]{\includegraphics[width=100mm,height=150mm]{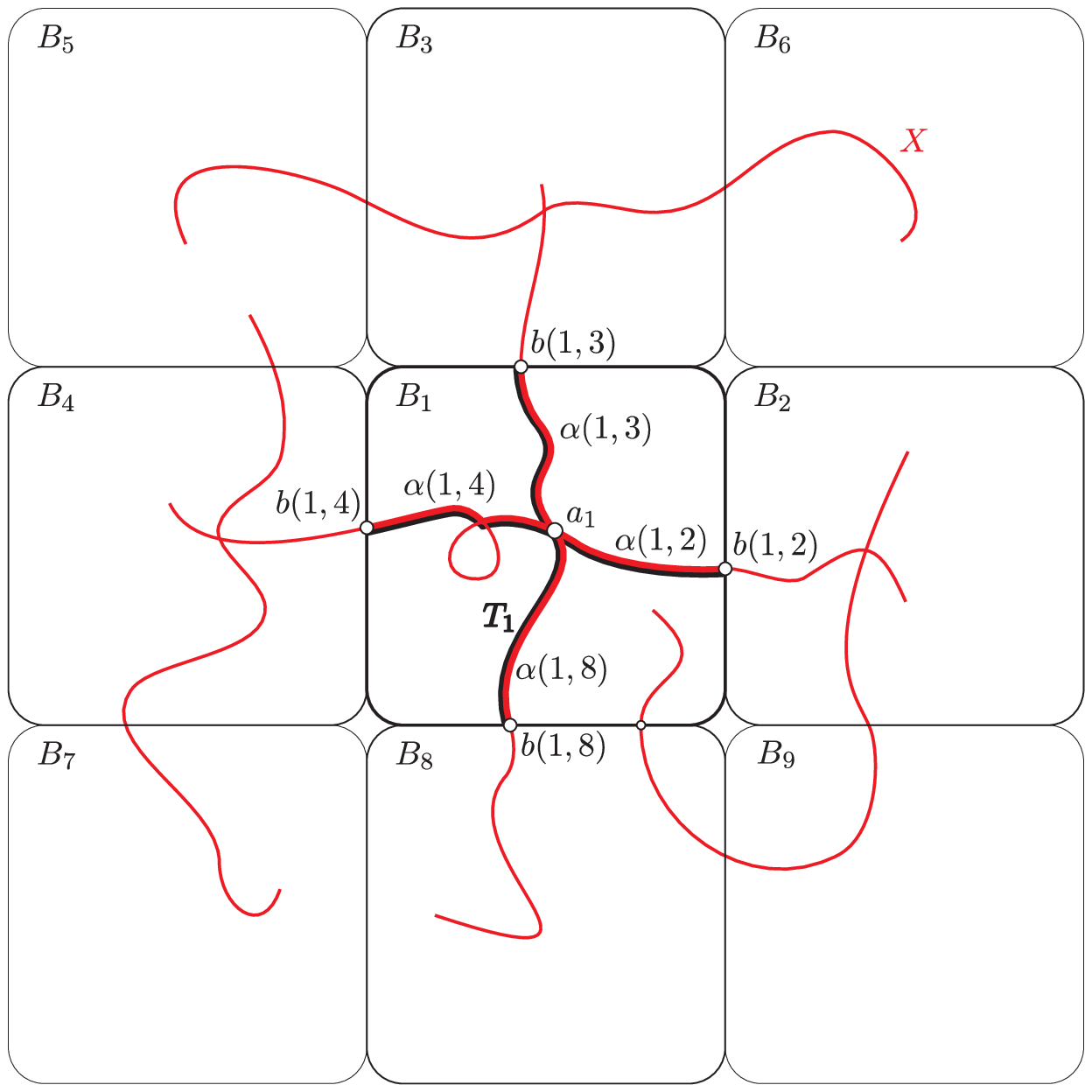}}
\mbox{\null}\hfill%
\begin{minipage}{12cm}
\caption{Brick partition $\{B_1,B_2,\dots,B_9\}$ of $X$ with points $a_1$, $b(1,j)$, arcs $\alpha(1,j)$ and a tree $T_1$.}\label{Fig3}
\end{minipage}
\hfill\mbox{\null}
\end{figure}

\

Now define a partial retraction  $r_i: T_i \cup \bd B_i \to T_i$ by putting $r_i | T_i =id$ and sending  $\bd B_i \cap \bd B_j$ onto $\{b(i,j)\}$, for all pairs   $(i, j)$  of dependent $i,j$.
   Let $\overline{r_i}$ be a continuous extension of $r_i$ over $B_i$ (notice that $T_i$ is an $AR$). Then the union of all  $\overline{r_i}$'s,  $i = 1, 2, \dots , n$ is the required retraction  $X$ onto the graph $G = \bigcup_{i=1}^n T_i$.
   \end{proof}

\begin{proof}[Proof 2 of Theorem~\ref{t:LCcurves}]
 Graphs belong to class 0-$CR$, as they are polyhedra. So, it remains to apply Theorems~\ref{t:poly}, \ref{t:curve retr} and~\ref{t:approx}.
 \end{proof}

\

\section{More examples of  0-$CR$-continua}\label{sec:examples}
The following general theorem is a rich source of such examples.

\begin{theorem}\label{t:products}
If $\mathcal{X}$ is a class of compact spaces that admit $\epsilon$-retrac\-tions onto   polyhedra,  for each  $\epsilon > 0$,  then
  \begin{enumerate}
        \item the product (finite or infinite) of spaces from  $\mathcal{X}$ belongs to 0-$CR$,
        \item the cone and suspension over  $X \in \mathcal{X}$ belongs to  0-$CR$.
  \end{enumerate}
\end{theorem}

The theorem easily follows  from Theorems~\ref{t:approx},~\ref{t:poly} and the following well known facts:
\begin{enumerate}
\item
the finite product of polyhedra, the cone and suspension over a polyhedron are polyhedra,
\item
 the infinite product of polyhedra admits $\epsilon$-retractions onto  finite products  of polyhedra for every $\epsilon > 0$.
 \end{enumerate}

 In particular,  we have the corollary.
 \begin{corollary}\label{c:products} The following spaces belong to class  0-$CR$:
 \begin{enumerate}
\item
 products  (finite or infinite) of locally connected curves, cones and suspensions over the products,
\item
compact Hilbert cube manifolds (since they are homeomorphic to products of a polyhedron and the Hilbert cube)
\item
Menger continua $M^n_k$ (see their description in~\cite{Eng}), their products, cones and suspensions over the products.
\item
all known dendroids, their products, cones and suspensions over the products (according to R. Cauty~\cite{Cauty}, all dendroids admit arbitrarily small retractions onto trees),
\item
all indecomposable continua  homeomorphic to inverse limits of arcs  with open bonding maps  (so called the simplest indecomposable continua), their products, cones and suspensions over the products.
\end{enumerate}
\end{corollary}

\

\section{Perfectness and density of periodic points in generic $CR(f)$}\label{per}

Many properties of a generic  map on  PL-manifolds are presented in~\cite{Hur2},~\cite{AHK}. In this section
   we show that some of them hold for much wider classes of compacta. Specifically, we are going to show that, for a generic $f$,  $CR(f)\simeq\mathcal C$ and the set of periodic points of $f$ of periods greater than any given integer $l\ge 2$ is dense in $CR(f)$ if $f$ is defined    on polyhedra, compact  Hilbert cube manifolds, local dendrites and finite products of them. Recall that a locally connected continuum $X$ is a \emph{local dendrite}, if $X$  has a base of open subsets $B$ such that $\overline{B}$ is a dendrite.   In general, the key properties of such spaces  responsible for detecting and multiplying periodic points are the local periodic point property and  mapping extension property.

\begin{definition}
A space $X$ has the \emph{weak local periodic point property} ($X\in wLPPP$) if there is an open base  $\mathcal{B}$ such that, for each $B \in \mathcal{B}$ and each continuous map  $f:X\to X$, whenever $f(\overline{B})\subset B$ then $f$ has a periodic point in $B$.

 Similarly, we can define the \emph{weak local fixed point property} ($wLFPP$).
\hfill$\square$
\end{definition}

Clearly,  $wLFPP\subset wLPPP$. The following easy examples show that the inclusion is proper.
\begin{example}\label{e:LPPP}\hfill

1. Let $X = \{0,1,\frac{1}{2},\dots\}$  be the subspace of the Euclidean line and $\mathcal B$ be an open base in $X$. It is easy to see that $X\in wLPPP$ and  for any set $U \in \mathcal B$  containing $0$, one easily finds a map $f:X\to X$ with $f( \overline{U} )\subset U$ without fixed points in $U$.

2. Consider the product  $Y=  X \times [0,1]$. The products $U\times V$, where $U\in \mathcal B$ and $V$ is an interval open in $[0,1]$, form a base $\mathbb{B}$ for  $wLPPP$ (use the fixed point property of $\overline V$ and $wLPPP$ of $X$).  To see that $Y\notin wLFPP$ take arbitrary open neighborhood $W$ of point $(0,0)$ in $Y$. There exist $U\in \mathcal B$ and an interval $V=[0,a)$ such that $\overline{U}\times \overline{V}\subset W$. Take a retraction $r:Y\to \overline{U}\times \overline{V}$. Then $(f\times id_{\overline V})r:Y\to W$ maps $\overline{W}$ into $W$ and  has no fixed point in $W$.

3. In a similar way we show that the cone $Z=Y/_{(X\times\{1\})}$ has $wLPPP$ but $Z\notin wLFPP$.
\hfill $\square$
\end{example}

\

\begin{remark}
Relationships between  properties $wLFPP$, $wLPPP$ and their stronger versions  $LFPP$ and $LPPP$ for continua are studied in~\cite{IK}.
\end{remark}

\

We will use the following well-known  mapping extension theorems (see~\cite[Theorem 3.1, p. 103]{Bor} and~\cite[Proposition 2.1.4]{Best}).
\begin{lemma}\label{l:ext}
Let a compact space $X$ be an ANR or an $n$-dimensional $LC^{n-1}$-space. Then for each $\epsilon>0$ there is $\eta>0$ such that, for each closed subset $A$ of $X$ and any two continuous maps
$f,g:A\to X$ such that $\hat{d}(f,g)<\eta$, if $f$ extends to a map $f':X\to X$, then $g$ extends to a map $g':X\to X$ with  $\hat{d}(f',g')<\epsilon$.
\end{lemma}

\

Denote
$$Per_{l}(f)=\{x\in Per{f}: \text{the $f$-orbit of $x$ has cardinality  $\ge l$}\},$$
\begin{multline*}
{\mathcal I}_\epsilon=\\
\{f\in C(X,X):card(B_\epsilon(x)\cap CR(f))>1 \mbox{ for each } x\in CR(f)\},
\end{multline*}
\begin{multline*}
\mathcal{I}_{\epsilon,l}=\\
 \{f\in C(X,X):card(B_\epsilon(x)\cap Per_{l}(f))>1 \mbox{ for each } x\in CR(f)\},
 \end{multline*}
where $B_\epsilon(x)$ is the $\epsilon$-ball around $x$.

\

Clearly, we have
$$\bigcap_n {\mathcal I}_{\frac1n} = \{f\in C(X,X):CR(f) \mbox{ has no isolated point}\},$$
and
\begin{multline*}
\bigcap_n \mathcal {I}_{\frac1n,l,} =\\
 \{f\in C(X,X):\text{$Per_{l}(f)$ has no isolated point and is dense in $CR(f)$}\}.
 \end{multline*}

\

\begin{proposition}\label{prop:i}
Let a perfect compact space $X$ be an ANR or an $n$-dimensional $LC^{n-1}$-space. Then:
\begin{enumerate}
\item
the sets
$\inte({\mathcal I}_\epsilon)$ and
$\mathcal {I}_{\epsilon,l}$
are  dense in $C(X,X)$ for each $\epsilon>0$ and $l\ge 2$;
\item
if $X\in wLPPP$, then also
$\inte(\mathcal {I}_{\epsilon,l})$ is dense in $C(X,X)$.
\end{enumerate}
\end{proposition}

\begin{proof}
Fix $f\in C(X,X)$ and $\epsilon,\delta>0$.
We are going to find $g'\in C(X,X)$ and $\gamma>0$ such that, for each $h\in C(X,X)$ with $\hat{d}(h,g')<\gamma$, we have $\hat{d}(h,f)<\delta$ and $h\in{\mathcal I}_{\epsilon}.$

By the upper semi-continuity of $CR$, there is $0<\delta'<\delta$ such that
$$\text{if }\quad \hat{d}(f,h)<\delta',
\quad\text{then}\quad CR(h)\subset B_{\frac{\epsilon}4} (CR(f)),$$
where $B_\alpha(A)$ denotes the $\alpha$-ball around $A\subset X$.

Choose  $0<\eta< \frac{\epsilon}4$  as in Lemma~\ref{l:ext}, i.~e.
if $g:A=\overline A\to X$ and $\hat{d}(f|A,g)<\eta$, then there is an extension
$g'\in C(X,X)$ of $g$ with $\hat{d}(f,g')<\delta'.$

Take an $\eta$-dense subset $S=\{x_1,\ldots,x_k\}$ of $CR(f)$,
i.~e. $B_\eta( S) \supset CR(f)$.

There is $0<\lambda<\frac{\eta}3$  such that
$$\text{if}\quad  d(x,x')<\lambda \quad\text{then}\quad d(f(x),f(x'))<\eta{/}3.$$

Let $(x_i^j)_{0\le j\le n_i}$ be $\lambda$-chains for $f$, $i=1,\ldots,k$,
i.~e.
$$x_i^0=x_i, \quad f(x_i^j)\in B_{\lambda}(x_i^{j+1})\quad\text{if}\quad j<n_i\quad\text{and}\quad f(x_i^{n_i})\in B_{\lambda}(x_i^{0}).$$

For each $i=1,\ldots,k$, $j=0,1,\ldots,n_i$, $l\ge 2$, $r=0,1,\ldots,l-1$,
choose nonempty open subsets
$U_i^j(r)$ of $B_{\lambda}(x_i^j)$
with mutually disjoint closures $\overline{U_i^j(r)} \subset B_{\lambda}(x_i^j)$. Pick a point $p_i^j(r)$ in each $U_i^j(r)$ and let $$p_i^{n_i+1}(r)=p_i^0(r+1\!\!\!\!\pmod l)\quad\text{and}\quad U_i^{n_i+1}(r)=U_i^0(r+1\!\!\!\!\pmod l).$$

Note that if $x\in \overline{U_i^j(r)}$,  then
$d(f(x),p_i^{j+1}(r))<\eta$.

Define
$$g:\bigl( X\setminus \bigcup_{i,j}B_{\lambda}(x_i^j)\bigr)
\cup \bigcup_{i,j,r}\overline{ U_i^j(r)} \to X$$
by
$$
g(x)=
\begin{cases}
f(x)& \text{for $x\in X\setminus \bigcup_{i,j}B_{\lambda}(x_i^j)$},\\
p_i^{j+1}(r)& \text{for $x\in \overline{U_i^j(r)}$}.
\end{cases}
$$
Since $g$ is $\eta$-close to $f$ (on the domain of $g$), it extends to a map
 $g':X\to X$ which is $\delta'$-close to $f$.
 Observe that  $p_i^0(r)\in Per(g')$, the orbit of  $p_i^0(r)$ under $g'$ has length $l(n_i+1)\ge l$ and  $p_i^0(r)\neq p_i^0(r')$  if $r\neq r'$.  We have
 $$S\subset B_\lambda\bigl(\{p_i^0(r): \text{$i\le k$ and  $r < l$}\}\bigr)$$
 and
 $$ Per(g')\subset CR(g')\subset B_{\frac{\epsilon}4+\eta}(S),$$
 therefore $g'\in \mathcal I_{\epsilon,l}$.
 This proves the second statement in (1).

 By the compactness, there is $\beta>0$ such that if $h:X\to X$ is $\beta$-close to $g'$, then
$$h(\overline{U_i^j(r)})\subset U_i^{j+1}(r)\quad\text{for all $i,j,r$}.$$
Therefore, for  $h:X\to X$ with $\hat{d}(h,g')<\beta$, we have
$$h^{l(n_i+1)}(\overline {U_i^j(r)})\subset U_i^j(r).$$
Since the chain recurrent set of any self-map is always nonempty,
$$CR(h^{l(n_i+1)}|\overline {U_i^j(r)})\neq\emptyset.$$
Hence
$CR(h)\cap \overline {U_i^j(r)}\neq\emptyset$.

If $X\in wLPPP$, then choosing $U_i^j(r)$ from the base $\mathcal B$ for $wLPPP$, we get a periodic point of $h$ with period at least $l(n_i+1)$ in each
$U_i^j(r)$.

Finally, for $0<\gamma<\beta$ and  $\gamma<\delta' - \hat{d}(f,g')$,
let us consider $h$ with  $\hat{d}(h,g')<\gamma$.
Obviously,  $\hat{d}(h,f)<\delta$.
Since $\hat{d}(h,f)<\delta'$,  we have
$$CR(h)\subset B_{\frac{\epsilon}4}(CR(f)).$$
Therefore, for $x\in CR(h)$ and $r<l$, there exist
$$z\in CR(f),\quad x_i\in S,\quad\text{and}\quad
x(i,r)\in U_i^0(r)\cap CR(h)$$
such that
\begin{multline}\tag{$*$}\label{e:dense}
d(x,x(i,r))\leq d(x,z)+d(z,x_i)+d(x_i,p_i^0(r))+d(p_i^0(r),x(i,r))< \\
\frac{\epsilon}4+\frac{\epsilon}4+\lambda+2\lambda<\epsilon
\end{multline}
Since $x(i,r)\not=x(i,r')$ for $r\neq r'$, we have $x(i,r)\not=x$ or $x(i,r')\not=x$, so $h\in {\mathcal I}_{\epsilon}$.

If $X\in wLPPP$, then in~\eqref{e:dense} we take
$x(i,r)\in U_i^0(r)$ with period at least $l(n_i+1)$,
 so $h\in {\mathcal I}_{\epsilon,l}$. Then~\eqref{e:dense}  also yields that $Per_{l}(h)$ is $\epsilon$-dense in  $CR(h)$.
\end{proof}

\

An immediate consequence of Proposition~\ref{prop:i} is the following theorem.

\begin{theorem}\label{t:perf}
Let a perfect compact space $X$ be an ANR or an $n$-dimensional $LC^{n-1}$-space. Then:
\begin{enumerate}
\item
 a generic $f\in C(X,X)$
has  perfect $CR(f)$;
\item
if $X\in wLPPP$, then, for any $l\ge 2$,  a generic map $f\in C(X,X)$ has perfect
$CR(f)$ with
$\overline{Per_{l}(f)}=CR(f)$.
\end{enumerate}
\end{theorem}

 \begin{corollary}\label{cor:Cantor}
Let $X$ be as in Theorem~\ref{t:perf}. If $X\in$ 0-$CR$, then the set $$\{f\in C(X,X): CR(f)\simeq\mathcal C\}$$
is residual in $C(X,X)$. In particular, if $X$ is a polyhedron, a local dendrite or a finite product of polyhedra and local dendrites or a Hilbert cube manifold, then, for $l\ge 2$,
$$\{f\in C(X,X): \overline{Per_{l}(f)}=CR(f)\simeq\mathcal C\}$$
is residual in $C(X,X)$.
\end{corollary}

\section{Measure-theoretic generic property}\label{measure}

 The following lemma and theorem were proved in~\cite{Y}.

 \begin{lemma}\label{l:Y}
 For any compact space $X$  and a finite Borel measure $\mu$ on $X$, the set $\{ f \in C(X,X): \mu (CR(f)) = 0  \}$ is a $G_\delta$-subset of $C(X,X)$.
\end{lemma}

\begin{theorem} [{\cite[3.1]{Y}}]\label{t:Y}
  If $X$ is a compact  ANR-space (or a compact $n$-dimensional $LC^{n-1}$-space) and $\mu$ is a finite Borel measure on $X$ without atoms at isolated points, then the set
$$\{ f \in C(X,X): \mu (CR(f)) = 0  \}$$
is dense $G_\delta$ in $C(X,X)$.
\end{theorem}

Consider the class  $\mu CR$ of compact spaces $X$ such that for any finite, atom-less, Borel measure $\mu$ on $X$ the set
$$\{ f \in C(X,X): \mu (CR(f)) = 0  \}$$
is dense $G_\delta$ in $C(X,X)$.

 \

\begin{theorem}\label{t:retrmeasure}
Let $X$ be a compact space that admits a retraction $r_\epsilon :X \rightarrow Y_\epsilon$, for each  $\epsilon >0$, such that  $\hat{d}(r_\epsilon, id) < \epsilon$ and  $Y_\epsilon=r_\epsilon(X)\in\mu CR$. Then  $X\in\mu CR$.
\end{theorem}

 The above theorem can be proved in the same way as Theorem~\ref{t:approx}. It suffices to note that the induced measure $\mu|_{Y_{\gamma}}$ on $Y_{\gamma} \subset X$ is finite, Borel  and atom-less.

\

 As a consequence, class $\mu CR$ contains all  ANR-compacta, all $n$-dimensional $LC^{n-1}$-compacta and all compacta that can be $\epsilon$-retracted onto them for every $\epsilon>0$. In particular, all spaces mentioned in Section~\ref{sec:examples} are in $\mu CR$.

 \

 \section{Genericity of zero-dimensional maps}
 For many spaces considered in previous sections generic maps turn out to be also zero-dimensional. In this section we recall and extend some known facts about these phenomena.

 A map $f:X\to Y$ is $n$-dimensional if its fibers $f^{-1}(y)$, $y\in Y$, are at most $n$-dimensional (here we consider the covering dimension).
Denote by $C_n(X,Y)$ the set of all  $n$-dimensional maps from $X$ to $Y$.

We say that a compact space $X$ is in  class  0-$DIM$ if  $C_0(X,X)$ is dense $G_{\delta}$  in $C(X,X)$.

The following lemma is well known~\cite[Theorem 5, p. 109]{Kur}.

\begin{lemma} \label{le12}
If $X$, $Y$ are compact spaces, then the set
$C_0(X, Y)$
is $G_{\delta}$ in the mapping space $C(X, Y)$.
\end{lemma}

As in the cases of classes 0-$CR$  and   $\mu CR$, we have an approximation theorem for 0-$DIM$. This time we do not need to have small retractions---small mappings onto subspaces suffice.

\begin{theorem}\label{t:approx dim}
If $X$ is a compact space admitting, for each $\delta >0$, a map  $p_\delta :X \rightarrow Y_\delta$ such that $\hat{d}(p_\delta, id) < \delta$ and $Y_\delta = p_\delta(X) \subset X$ belongs to class 0-$DIM$, then $X\in$  0-$DIM$.
\end{theorem}

\begin{proof}
We show that the set
\begin{multline*}
A_n=\\
\{f \in C(X,X): \text{components of fibers of $f$ have diameters $< 1/n$} \}
\end{multline*}
is open and dense in $C(X,X)$ and notice that $C_0 (X,X) = \bigcap_{n=1}^{\infty} A_n$.

The openness is  routine to check. The density can be shown as in the proof of Theorem~\ref{t:approx}.
\end{proof}

It is well known that PL-manifolds are in class  0-$DIM$ (see, e.g., \cite{A}).
The simplest spaces that belong to  0-$DIM$ are graphs. Theorem~\ref{t:approx dim} and Mazurkiewicz theorem~\cite{Maz} mentioned in Section~\ref{LC} (or Theorem~\ref{t:curve retr}) allow to include all locally connected curves to 0-$DIM$.  This fact also follows from a more general result on Lelek maps~\cite[2.7]{KM}. It was proved in~\cite{KM} that class 0-$DIM$ contains all Menger manifolds and $n$-dimensional ANR-compacta with  piecewise embedding dimension $ped =n$. We will not work with $ped$ here. Instead,  extending methods from~\cite{Lev}, we are going to  show that finite products of graphs are in 0-$DIM$.  This needs the notion of a Lelek map.

A map $f: X \rightarrow Y$ between compacta is an $n$-{\it dimensional Lelek map} if the dimension of the union of all nondegenerate components of all fibers of $f$ is $\leq n$. The set of Lelek maps from $X$ to $Y$ is denoted by  $L_n(X,Y)$. If $f\in C(X,Y)$ and $a >0$, then put
\begin{multline*}
F(f,a)=\\
\bigcup\{B\subset X: \text{$B$ is a component of $f^{-1}(y)$, $y\in Y$, with $\diam(B) \geq a$} \} ,
\end{multline*}
and  $F(f)=\bigcup_{i=1}^{\infty} F(f, 1/i)$. Thus,
$$f\in L_n(X,Y) \quad\text{if and only if}\quad \dim F(f) \leq n.$$

It is known that
\begin{itemize}
\item
$n$-dimensional Lelek maps are  $n$-dimensional,
\item
$L_0(X,Y)=C_0(X,Y)$,
\item
 $L_n(X,Y)$ is a $G_{\delta}$ subset of $C(X,Y)$
 \end{itemize}
 for any compact $X,Y$~\cite[Proposition 2.1]{Lev}.

\

The following theorem was proved in~\cite[Proposition 2.2]{Lev} for $Y$ the unit interval. By an appropriate modification of its proof, we get it for any $LC^n$-continuum. We attach a proof for  reader's convenience.

\begin{theorem}\label{Lelekdensity}
If $X$ is a compact  $(n+1)$-dimensional space and  $Y$ is a nondegenerate $LC^n$-continuum, then the set $L_n(X,Y)$ is dense $G_{\delta}$ in $C(X,Y)$.
\end{theorem}

\begin{proof}
It is enough to prove the density.

Fix a $0$-dimensional closed subset $Z$ of $X$. Denote
$$H(Z,\eta)=\{f \in C(X,Y):F(f,\eta) \cap Z = \emptyset   \}.$$
It is straightforward to check that  $H(Z,\eta)$ is open in $C(X,Y)$.

We are going to show that $H(Z,\eta)$ is dense in $C(X,Y)$. So, let $f \in C(X,Y)$ and $\epsilon >0$. Choose $\delta$ for $\epsilon$ as in Lemma~\ref{le0} and let $\{ V_1, V_2, \dots V_k \}$  be a cover of $Z$ by open in $X$, mutually disjoint subsets such that $\diam(V_i)< \eta$,  $\diam f(\overline{V_i})<\delta /2$ for $i= 1, 2, \dots , k$. In each $V_i$ choose an open subset $W_i$ such that $Z \cap W_i = Z \cap V_i$, $\overline{W_i} \subset V_i$.

Construct maps $g_i: \overline{V_i} \rightarrow Y$ satisfying
\begin{itemize}
\item
$\hat{d}(f, g_i) < \epsilon$ on $\overline{V_i}$,
\item
 $g_i|\bd V_i = f|\bd V_i$,
 \item
 $g_i(\bd W_i)$ is a singleton,
 \item
 $g_i(W_i)\cap g_i(\bd W_i) = \emptyset$
 \end{itemize}
as follows. Fix  $z_i \in f(\bd W_i)$ and an arc  $L_i$ with its end-point at $z_i$ and with $\diam L_i\leq \delta / 2$.
  Let $g_i(\bd W_i) = \{z_i\}$,  $g_i(W_i) \subset L_i \setminus \{z_i\}$ and  $g_i|\bd V_i = f|\bd V_i$. By Lemma~\ref{le0}, extend $g_i$ over $\overline{V_i}$. Then  extend $\bigcup_{i=1}^k g_i$ to $g:X\to Y$ by putting $g|X \setminus \overline{\bigcup V_i}= f|X \setminus \overline{\bigcup V_i}$.

 Each continuum  of diameter $\geq \eta$,  contained in a fiber of $g$, can not intersect any $W_i$, since otherwise it would intersect $\bd W_i$ which is impossible. So,   $g \in H(Z,\eta)$ which proves the density of  $H(Z,\eta)$.

 Now, there are  $0$-dimensional closed subsets  $Z_i \subset X$ such that $\dim (X \setminus\bigcup_{i=1}^\infty  Z_i) \leq n$.
The intersection  $H=\bigcap_{i,j} H(Z_i, 1/j)$ is dense in $C(X, Y)$ and $H\subset L_n(X,Y)$.
\end{proof}

\

In~\cite[p. 260]{Lev} M. Levin  roughly sketched an idea how to prove that the set $L_{n-k}(X,I^k)$ is dense  $G_{\delta}$ in $C(X,I^k)$ for an $n$-dimensional compact space $X$. In particular, $C_0(I^n,I^n)$ is dense  $G_{\delta}$ in $C(I^n,I^n)$, i.~e., $I^n\in$  0-$DIM$.
We use the Levin's idea and provide a more detailed proof that products of graphs are in  0-$DIM$.

\begin{theorem}\label{t:graphs}
If $X=G_1 \times \dots \times G_n$ and each $G_i$ is a graph   then $X\in$  0-$DIM$.
\end{theorem}
\begin{proof}
In view of Lemma~\ref{le12}, it is enough to show the density of  $C_0(X,X)$ in $C(X,X)$.

Fix a map $f=(f_1, \dots f_n)$,  $f_i:X \rightarrow G_i$.
 Since $L_{n-1}(X,G_1)$ is dense in $C(X,G_1)$ (Theorem~\ref{Lelekdensity}) we can approximate  $f_1$ by a map $\bar{f}_1 \in L_{n-1}(X,G_1)$. We have $\dim F(\bar{f}_1) \leq n-1$.
The set  $F(\bar{f}_1)$ is a  union of compact sets
$$K_j=F(\bar{f}_1, 1 / j), \ j=1, 2, \dots.$$
Observe that, similarly as $L_{n-2}(X,G_2)$, the set
 $$\mathcal{K}_j=\{g \in C(X,G_2):\dim F(g|K_j)\leq n-2 \}$$
is $G_\delta$ in $C(X,G_2)$. It is also dense in  $C(X,G_2)$. Indeed, for each  $g \in C(X,G_2)$ and  $\epsilon >0$,   find a map $\bar{g} \in L_{n-2}(K_j, G_2)$ which is $\epsilon$-close to  $g|K_j$ (by Theorem~\ref{Lelekdensity}).   Graph $G_2$ being an $ANR$, we can assume that $g|K_j$ and $\bar{g}$ are  $\epsilon$-homotopic. By the Borsuk homotopy extension theorem~\cite[Theorem 8.1, p. 94]{Bor}, extend $\bar{g}$ to a map $\bar{g}^* \in C(X, G_2)$ which is $\epsilon$-homotopic to $g$.
Hence, the set $\mathcal{K}=\bigcap_{j=1}^\infty \mathcal{K}_j$ is dense $G_\delta$ in $C(X,G_2)$. Since
$$\mathcal{K}=\{f \in C(X,G_2): \dim F(f|F(\bar{f}_1))\leq n-2 \},$$
we can choose $\bar{f}_2 \in \mathcal{K}$ arbitrarily close to $f_2$, etc. Continuing inductively, for  $i=1, \dots , n-1$, we get an approximation $\bar{f}_{i+1}:X \rightarrow G_{i+1}$ of  $f_{i+1}:X \rightarrow G_{i+1}$ such that
$$\dim F\bigl(\bar{f}_{i+1}|F(\bar{f}_i|F(\bar{f}_{i-1} \dots \bigr) \leq n-i-1 $$
The map $\bar{f}=(\bar{f}_1, \dots , \bar{f}_n)$ approximates $f$ and is  zero-dimensional, because
$$\bar{f}^{-1}(y)=\bigcap_{i=1}^n \bar{f}_i^{-1}(y),$$
so if  $C$ is a nondegenerate component of  $\bar{f}^{-1}(y)$, then
$$C \subset F(\bar{f}_1) \cap F\bigl(\bar{f}_2|F(\bar{f}_1)\bigr)  \cap \dots \cap F\bigl(\bar{f}_{n}|F(\bar{f}_{n-1}|F(\bar{f}_{n-2} \dots \bigr),$$
where the last set in the above intersection is zero-dimensional, a contradiction.
\end{proof}

As a corollary from Theorems~\ref{t:approx dim}, \ref{t:graphs} and~\ref{t:curve retr} we get

\begin{theorem}\label{t:0DIMproducts}
Any product (finite or infinite) of locally connected curves belongs to class 0-$DIM$.
\end{theorem}

\

\section{Examples of non-0-$CR$-continua}

The simplest such examples  are one-dimensional, indecomposable, non-planar continua $M_1$, $M_2$ of Cook~\cite{Cook}.  They are not in class 0-$CR$ because they are nondegenerate and rigid, i.~e.,  the only continuous self-maps of $M_i$ are constant maps and the identity map.

T. Ma\'{c}kowiak constructed in~\cite{Mack1},  \cite{Mack2} rigid continua that are hereditarily decomposable and arc-like, so they are planar.

Other examples in~\cite{Cook} of one-dimensional continua $H_n$, $n\in\mathbb N$, that admit only $n$   non-trivial continuous self-maps also do not belong to 0-$CR$.

Clearly, all these examples contain no arcs.

\

We are now going to describe $\frak c$-many simple examples of one dimensional continua with 3 arc-components which are uncomparable by continuous surjections (there are no continuous surjections between two distinct such continua)  and which do not belong to 0-$CR$. Let us first consider  a piecewise linear function $l: [-1,0) \rightarrow [0,1]$ with the sets of local minima and of local maxima equal to the set of rationals in $I$.  If  graph $G(l)$ is sufficiently complicated (``crooked everywhere" in order to guarantee that  $A= \{0\} \times [0,1] \subset Fix(f)$ for any $f:\overline{G(l)} \to \overline{G(l)}$  close enough to the identity), then we also get a compactification $\overline{G(l)}$ of a ray by an arc which is not in   0-$CR$.
We get our examples by  adding to $\overline{G(l)}$ uncomparable compactifications $K_\gamma$, $\gamma< \frak c$,  of a ray $(0,1]$ with remainder $A$ as constructed in ~\cite{Aw}.

Joining the end-points of  rays $(0,1]$ and $G(l)$ with point $\{(0,0)\}$ by  arcs, we get  arcwise connected planar curves which are not in    0-$CR$.

\

Concerning higher dimensional examples,
E. van Douwen observed in~\cite{vD} that  an infinite product $X=\prod_{n=1}^\infty X_n$ of subcontinua  $X_n$ of the Cook's continuum $M_1$ has the property that each map $f:X \rightarrow X$ is of the form  $\prod_n f_n$, where  $f_n:X_n \rightarrow X_n$ is a a constant map or the identity. It follows that $X\notin$ 0-$CR$. As an infinite product of one-dimensional continua,  $X$ is  infinite dimensional.

Since for every $\epsilon>0$ there is $k$ such that $X$ can be $\epsilon$-retracted onto  $\prod_{n=1}^k X_n$ (by the projection), it follows from
Theorem~\ref{t:approx} that, for sufficiently large $k$, the product $\prod_{n=1}^k X_n$ is not in  0-$CR$. Thus we also have finite-dimensional examples of arbitrarily large dimension.

\

\section{Final remarks and  questions}

The following corollary combines  results of previous sections and provides the strongest properties of a generic map.

\begin{corollary}\label{t:concl}
Let $X$ be a perfect compact ANR-space (or a compact $n$-dimensional $LC^{n-1}$-space) satisfying $wLPPP$,  $X\in$ 0-$CR$ and   $X\in$ 0-$DIM$ and let $\mu$ be a finite,  atom-less, Borel measure on $X$. Then
\begin{multline}\label{eq:concl}
CR(f)=\overline{Per(f)}\simeq\mathcal{C},\quad \mu \bigl(CR(f)\bigr) = 0 \\
\text{and  $f$ is zero-dimensional for a generic $f\in C(X,X)$}.
\end{multline}
In particular,~\eqref{eq:concl} holds for polyhedra $P$ with $ped(P)=\dim(P)$ and their products with a Hilbert cube,  PL-manifolds  and finite products of local dendrites.
\end{corollary}

\

\begin{question}
Do the following spaces  belong to 0-$CR$:
\begin{enumerate}
\item
locally connected continua of dimensions $\ge 1$,
\item
$LC^n$-continua, $n\ge 1$, or ANR-continua,
\item
compact topological manifolds (known results~\cite{AHK} concern smooth or PL-manifolds),
\item
solenoids,
\item
the pseudoarc,
\item
products, cones, suspensions over 0-$CR$-compacta ?
\end{enumerate}
\end{question}

\begin{question}
Let $X$ be  the cone over the harmonic sequence (over the Cantor set $\mathcal{C}$). Is $CR(f)\simeq \mathcal{C}$  for a generic $f$?
\end{question}

\bibliographystyle{amsplain}

\end{document}